\documentclass[reqno]{amsart}
\usepackage[english]{babel}
\usepackage{amssymb,amsmath,hyperref}
\usepackage{bbold}
\usepackage{amsrefs}
\usepackage{stackrel}
\usepackage[foot]{amsaddr}
\usepackage{mathrsfs}
\usepackage{cleveref, enumerate}

\newtheorem{thm}{Theorem}

\newtheorem{cor}[thm]{Corollary}
\newtheorem{defi}[thm]{Definition}
\newtheorem{rem}[thm]{Remark}
\newtheorem{nota}[thm]{Notation}

\newtheorem{princ}[thm]{Principle}

\newtheorem{ack}[thm]{Acknowledgement}

\newtheorem*{tempo*}{Template}

\newcommand\be{\begin{equation}}
\newcommand\ee{\end{equation}}

\newbox\gnBoxA
\newdimen\gnCornerHgt
\setbox\gnBoxA=\hbox{$\ulcorner$}
\global\gnCornerHgt=\ht\gnBoxA
\newdimen\gnArgHgt

\def\Godelnum #1{%
	\setbox\gnBoxA=\hbox{$#1$}%
	\gnArgHgt=\ht\gnBoxA%
	\ifnum \gnArgHgt<\gnCornerHgt
		\gnArgHgt=0pt%
	\else
		\advance \gnArgHgt by -\gnCornerHgt%
	\fi
	\raise\gnArgHgt\hbox{$\ulcorner$} \box\gnBoxA %
		\raise\gnArgHgt\hbox{$\urcorner$}}
\def\bdefi{\begin{defi}\rm}
\def\edefi{\end{defi}}
\def\bnota{\begin{nota}\rm}
\def\enota{\end{nota}}
\def\brem{\begin{rem}\rm}
\def\erem{\end{rem}}

\def\PFTPA{\textup{\textsf{PF-TP}}_{\forall}}

\def\FAN{\textup{\textsf{FAN}}}

\def\IST{\textup{\textsf{IST}}}
\def\ATR{\textup{\textsf{ATR}}}

\def\BCT{\textup{\textsf{BCT}}}
\def\CCT{\textup{\textsf{CCT}}}
\def\cont{\textup{\textsf{cont}}}

\def\RCAo{\textup{\textsf{RCA}}_{0}^{\omega}}

\def\WKL{\textup{\textsf{WKL}}}

\def\UWKL{\textup{\textsf{UWKL}}}

\def\T{\mathcal{T}}

\def\bye{\end{document}}

\def\N{{\mathbb  N}}
\def\Q{{\mathbb  Q}}
\def\R{{\mathbb  R}}

\def\HBU{\textup{\textsf{HBU}}}

\def\FAN{\textup{\textsf{FAN}}}

\def\MUC{\textup{\textsf{MUC}}}
\def\TJ{\textup{\textsf{TJ}}}

\def\WCN{{\textup{\textsf{WC-N}}}}
\def\SE{{\textup{\textsf{SE}}}}

\def\MP{{\textup{\textsf{MP}}}}
\def\WMP{{\textup{WMP}}}
\def\WMP{{\textup{\textsf{WMP}}}}

\def\R{{\mathbb{R}}}

\def\({\textup{(}}
\def\){\textup{)}}

\def\st{\textup{st}}

\def\asa{\leftrightarrow}

\def\di{\rightarrow}

\def\M{\textsf{\textup{DG}}}

\def\ACA{\textup{\textsf{ACA}}}
\def\paai{\Pi_{1}^{0}\textup{-\textsf{TRANS}}}

\def\QFAC{\textup{\textsf{QF-AC}}}

\def\PFTPE{\textsf{\textup{PF-TP}}_{\exists}}

\numberwithin{equation}{section}
\numberwithin{thm}{section}

\usepackage{comment}

\begin{document}
\title{A note on non-classical nonstandard arithmetic}
\author{Sam Sanders}
\address{School of Mathematics, University of Leeds, UK \& Depart of Mathematics, TU Darmstadt, Germamy}
\email{sasander@me.com}
\begin{abstract}
Recently, a number of formal systems for Nonstandard Analysis restricted to the language of finite types, i.e.\ \emph{nonstandard arithmetic}, have been proposed.  
We single out one particular system by Dinis-Gaspar, which is categorised by the authors as being part of \emph{intuitionistic nonstandard arithmetic}.  
Their system is indeed \emph{in}consistent with the \emph{Transfer} axiom of Nonstandard Analysis, and the latter axiom is classical in nature as it implies (higher-order) comprehension.  
Inspired by this observation, the main aim of this paper is to provide answers to the following questions:
\begin{enumerate}
\item[(Q1)] In the spirit of \emph{Reverse Mathematics}, what is the \emph{minimal} fragment of \emph{Transfer} that is inconsistent with the Dinis-Gaspar system?
\item[(Q2)] What other axioms are inconsistent with the Dinis-Gaspar system?
\end{enumerate}
Our answer to the first question suggests that the aforementioned inconsistency actually derives from the axiom of extensionality relative to the standard world, and that other (much stronger) consequences of \emph{Transfer} are actually harmless. 
Perhaps surprisingly, our answer to the second question shows that the Dinis-Gaspar system is inconsistent with a number of (non-classical) continuity theorems which one would -in our opinion- categorise as \emph{intuitionistic} in the sense of Brouwer.  
Finally, we show that the Dinis-Gaspar system involves a \emph{standard part map}, suggesting this system also pushes the boundary of what still counts as `Nonstandard Analysis' or `internal set theory'.  
\end{abstract}


\maketitle
\thispagestyle{empty}
\section{Introduction}
\subsection{Aim and motivation}\label{introke}
In the last decade, a number of versions of Heyting and Peano arithmetic in all finite types have been introduced (\cites{brie, fega,benno2, dinidini}) which are based on (fragments of) Nelson's \emph{internal set theory} (\cite{wownelly}).  
Such systems allow for the extraction of the (copious) computational content of Nonstandard Analysis, as discussed at length in \cite{SB}.  In this paper, we study the system $\M$ by Dinis-Gaspar to be found in  \cite{dinidini} and Section \ref{pewlim}; $\M$ has been described as follows:
\begin{quote}    
We present a bounded modified realisability and a bounded functional interpretation of \emph{intuitionistic} nonstandard arithmetic with nonstandard principles. (\cite{dinidini}*{Abstract}, emphasis added)
\end{quote}
Similar claims may be found in the body of \cite{dinidini}: $\M$ is part of intuitionistic mathematics, as claimed by the authors.  
This claim is not without merit: $\M$ is indeed inconsistent with the \emph{Transfer} axiom of Nonstandard Analysis, and this axiom is essentially the nonstandard version of 
comprehension.  By way of an example, \emph{Transfer} restricted to $\Pi_{1}^{0}$-formulas translates to the `Turing jump functional' $\exists^{2}$, as defined in Section \ref{TJF}.  
\emph{However}, `non-classical' does not necessarily imply `intuitionistic', and we shall study the following two questions in this paper.  
\begin{enumerate}
\item[(Q1)] In the spirit of \emph{Reverse Mathematics}, what is the \emph{minimal} fragment of \emph{Transfer} that is inconsistent with $\M$?
\item[(Q2)] What other (intuitionistic) axioms are inconsistent with $\M$?
\end{enumerate}
The main results of this paper constitute answers to the questions (Q1) and (Q2), and we now discuss them in some more detail. 

\smallskip

Regarding (Q1), we refer to \cites{simpson2, stillebron, simpson1} for an introduction and overview of Reverse Mathematics (RM for short).  We shall consider the \emph{parameter-free Transfer} principle studied in \cite{bennosam}, and related axioms.    
We will identify the axiom of extensionality (relative to the standard world) as the real culprit: this axiom follows from \emph{Transfer} and is inconsistent with $\M$, while other axioms implied by \emph{Transfer}, even involving the Turing jump functional, \emph{are} consistent with $\M$. 

\smallskip

Regarding (Q2), we show that $\M$ is inconsistent with a number of (non-classical) axioms 
which one would categorise as \emph{intuitionistic}, i.e.\ part of Brouwer's intuitionistic mathematics\footnote{Since `intuitionism' is the first keyword of \cite{dinidini}, we take `intuitionistic' to mean `part of Brouwer's intuitionistic mathematics'.  We discuss this choice in more detail in Remark \ref{jezus}.\label{haha}}.  The most blatant example is the statement that, relative to the standard world, all functionals
on the Cantor space are (uniformly) continuous.     
  
\smallskip  
  
In the course of investigating (Q1) and (Q2), one eventually stumbles upon the fact that the Dinis-Gaspar system allows one to define a (highly elementary)
\emph{standard part map}, as discussed in Section \ref{nsansa}.  Since such a map is not available in Nelson's internal set theory, 
and \emph{external} in Robinson's approach, the Dinis-Gaspar system thus pushes the boundary of what still counts as `Nonstandard Analysis'.

\smallskip  
  
As to the structure of this paper, we briefly discuss the importance of continuity in intuitionism in Remark \ref{intu}.  
The formal system $\M$ from \cite{dinidini} and associated prerequisites are sketched in Section \ref{pewlim}.  
Our main results may be found in Section~\ref{main}, 
which provide fairly definitive answers to questions (Q1) and (Q2).  We formulate the conclusion to this paper in Section \ref{konkelfoes}.  
 
\smallskip  

Next, we point out the (intimate) relationship between Brouwer 's intuitionistic mathematics and continuity, lest the reader believe the above is merely pedantry. 
\begin{rem}[Intuitionism and continuity]\rm\label{intu}
L.E.J.\ Brouwer is the founder of \emph{intuitionism}, a philosophy of mathematics which later developed into the first full-fledged school of \emph{constructive mathematics}.    
The latter is an umbrella term for approaches to mathematics in which `there exists $x$' is systematically interpreted as `we can compute/construct $x$' (and similarly for the other logical symbols).

\smallskip

Under this new interpretation of the logical symbols, certain laws do not make any sense, and are therefore rejected; the most (in)famous one being the \emph{law of excluded middle} $P\vee \neg P$.    
The resulting logic is \emph{intuitionistic logic}, and we refer to \cites{beeson1, troeleke1} for an introduction to the various approaches to constructive mathematics.   
  
\smallskip

Brouwer proved in 1927 (see \cite{vajuju}*{p.\ 444} for an English translation) that every total (in the intuitionistic sense) function on the unit interval is (uniformly) continuous, 
a result which seems to contradict classical mathematics.  The core axioms for intuitionistic mathematics indeed include a `continuity' axiom 
(called $\WCN$ in \cite{troeleke1} and $\textsf{BP}_{0}$ in \cite{beeson1}) which contradicts classical mathematics, 
and can be used to prove the aforementioned (uniform) continuity theorem by Brouwer.   

\smallskip

The previous is well-known, but is mentioned since we want to stress the following: a very low bar a logical system 
has to clear to deserve the moniker `intuitionistic', is to be consistent with the aforementioned continuity theorem and axiom.  
As it turns out, this does not seem to be the case for the Dinis-Gaspar system.  
\end{rem}
Finally, despite (or perhaps better: `because') the above criticism, we do believe that $\M$ has some has some interesting features.  
For instance, $\M$ is a sort of non-classical analogue of the Fernand-Oliva bounded functional interpretation (BFI hereafter; see \cite{BFI}). 
Now, the BFI \emph{refutes} extensionality, so results such as Theorem~\ref{skimitar} and its corollaries are, in this light, rather natural. 
Regarding BFI and related matters, \cite{BFI2}*{Section 3} is also highly informative.

\section{Preliminaries}\label{pewlim}
We introduce the Dinis-Gaspar system $\M$, and some preliminaries.  
\subsection{Internal set theory and its fragments} \label{otherm}
In this section, we discuss Nelson's \emph{internal set theory}, first introduced in \cite{wownelly}, and the Dinis-Gaspar system $\M$ from \cite{dinidini}.  
The system $\M$ is an \emph{extension} of a fragment of Nelson's system with (non-classical) axioms pertaining to \emph{majorizability}.  

\smallskip

In Nelson's \emph{syntactic} approach to Nonstandard Analysis (\cite{wownelly}), as opposed to Robinson's semantic one (\cite{robinson1}), a new predicate `st($x$)', read as `$x$ is standard' is added to the language of \textsf{ZFC}, the usual foundation of mathematics.  
The notations $(\forall^{\st}x)$ and $(\exists^{\st}y)$ are short for $(\forall x)(\st(x)\di \dots)$ and $(\exists y)(\st(y)\wedge \dots)$.  A formula is called \emph{internal} if it does not involve `st', and \emph{external} otherwise.   
The three external axioms schemes \emph{Idealisation}, \emph{Standard Part}, and \emph{Transfer} govern the new predicate `st';  they are respectively defined\footnote{The superscript `fin' in \textsf{(I)} means that $x$ is finite, i.e.\ its number of elements are bounded by a natural number.      } as:  
\begin{enumerate}
\item[\textsf{(I)}] $(\forall^{\st~\textup{fin}}x)(\exists y)(\forall z\in x)\varphi(z,y)\di (\exists y)(\forall^{\st}x)\varphi(x,y)$, for internal $\varphi$.  
\item[\textsf{(S)}] $(\forall x)(\exists^{\st}y)(\forall^{\st}z)\big([z\in x\wedge \varphi(z)]\asa z\in y\big)$, for any $\varphi$.
\item[\textsf{(T)}] $(\forall^{\st}x)\varphi(x, t)\di (\forall x)\varphi(x, t)$, where $\varphi$ is internal,  $t$ captures \emph{all} parameters of $\varphi$, and $t$ is standard.  
\end{enumerate}
The system \textsf{IST} is (the internal system) \textsf{ZFC} extended with the aforementioned three external axioms;  
the former is a conservative extension of \textsf{ZFC} for the internal language, as proved in \cite{wownelly}.    

\smallskip

In \cites{fega, dinidini, brie,benno2}, the authors study G\"odel's system $\textsf{T}$ extended with versions of the external axioms of \textsf{IST}.  
In particular, they consider nonstandard extensions of the (internal) systems \textsf{E-HA}$^{\omega}$ and $\textsf{E-PA}^{\omega}$, respectively \emph{Heyting and Peano arithmetic in all finite types and the axiom of extensionality}.       
We refer to \cite{brie}*{\S2.1} for the exact details of these (mainstream in mathematical logic) systems.  

\smallskip

The results in \cite{fega, dinidini} are inspired by those in \cite{brie}.  In particular, the notion of \emph{finiteness} central to the latter is replaced by the notion of \emph{strong majorizability}.  
The latter notion and the associated system $\M$ is introduced in the next paragraph, assuming familiarity with the higher-type framework of G\"odel's $\textsf{T}$.  

\smallskip

The system $\M$, a conservative extension of $\textsf{E-HA}^{\omega}$, is based on the Howard-Bezem notion of \emph{strong majorizability}.  We first introduce the latter and related notions.  
For a more extensive background on strong majorizability, see \cite{kohlenbach3}*{\S3.5}.    
\begin{defi}[Majorizability]\rm The strong majorizability predicate `$\leq^{*}$' is inductively defined as follows:
\begin{itemize}
\item $x \leq_{0}^{*} y$ is $x \leq_{0} y$;
\item $x \leq_{\rho\di \sigma}^{*} y$ is   $(\forall u)(\forall v\leq_{\rho}^{*}u)\big(xu\leq_{\sigma}^{*} yv \wedge yu\leq_{\sigma}^{*} yv  \big)  $.   
\end{itemize}
An object $x^{\rho}$ is called \emph{monotone} if $x\leq_{\rho}^{*}x$.  The quantifiers $(\tilde{\forall} x^{\rho})$ and $(\tilde{\exists} y^{\rho} )$ range over the monotone objects of type $\rho$, i.e.\ they are abbreviations for the formulas $(\forall x)(x\leq^{*}x \di \dots)$ and 
$(\exists y)(y\leq^{*}y\wedge\dots)$.
\edefi
The system $\M$ is defined as follows in \cite{dinidini}*{\S2}.  The language of $\textsf{E-HA}_{\st}^{\omega}$ is the language of $\textsf{E-HA}^{\omega}$ extended with a new `standardness' predicate $\st^{\sigma}$ for every finite type $\sigma$.  
The typing of the standardness predicate is usually omitted.    
\begin{defi}[Standard quantifiers]\label{notawin}\rm
We write $(\forall^{\st}x^{\tau})\Phi(x^{\tau})$ and $(\exists^{\st}x^{\sigma})\Psi(x^{\sigma})$ as short for 
$(\forall x^{\tau})\big[\st(x^{\tau})\di \Phi(x^{\tau})\big]$ and $(\exists^{\st}x^{\sigma})\big[\st(x^{\sigma})\wedge \Psi(x^{\sigma})\big]$.     
A formula $A$ is `internal' if it does not involve $\st$, and external otherwise.  The formula $A^{\st}$ is defined from $A$ by appending `st' to all quantifiers (except bounded number quantifiers).    
\end{defi}
Regarding the previous definition, we often say that `the formula $A^{\st}$ is the formula $A$ relative to the standard world'.
\bdefi[Basic axioms]\label{debs}
The system $ \textsf{E-HA}^{\omega}_{\st} $ is defined as $ \textsf{E-HA}^{\omega} + \T^{*}_{\st} + \textsf{IA}^{\st}$, where $\T^{*}_{\st}$
consists of the following axiom schemas.
\begin{enumerate}
\item[(a)] $x =_{\sigma} y \di (\st^{\sigma}(x) \di \st^{\sigma}(y));$
\item[(b)] $\st^{\sigma}(y) \di (x\leq_{\sigma}^{*} y \di \st^{\sigma}(x))$;
\item[(c)] $\st^{\sigma}(t)$, for each closed term $t$ of type $\sigma$; 
\item[(d)] $\st^{\sigma\di \tau} (z) \di (\st^{\sigma} (x) \di \st^{\tau}(zx))$.
\end{enumerate}
Items (a)-(d) are called the \emph{standardness axioms}, and (b) is singled out regularly below.  
The \emph{external induction axiom} \textsf{IA}$^{\st}$ is the following schema for any $\Phi$: 
\be\tag{\textsf{IA}$^{\st}$}
\Phi(0)\wedge(\forall^{\st}n^{0})(\Phi(n) \di\Phi(n+1))\di(\forall^{\st}n^{0})\Phi(n).     
\ee
\edefi
\noindent
The system $\M$ is then defined as $\textsf{E-HA}^{\omega}_{\st}$ plus the following non-basic axioms.  
\bdefi[Non-basic axioms]\label{nba}~
\begin{itemize}
\item Monotone Choice $ \textsf{mAC}^{\omega}$:     For any $\Phi$, we have   
\[
(\tilde{\forall}^{\st} x)(\tilde{\exists}^{\st}y)\Phi(x, y)\di (\tilde{\exists}^{\st}f)(\tilde{\forall}^{\st} x)(\exists y\leq^{*}f(x))\Phi(x,y). 
\]
\item Realization $\textsf{R}^{\omega}$:  For any $\Phi$, we have   
\[
(\forall x)(\exists^{\st}y)\Phi(x, y)\di (\tilde{\exists}^{\st}z)(\forall x)(\exists y\leq^{*}z)\Phi(x,y). 
\]
\item Idealisation $\textsf{I}^{\omega}$: For any internal $\phi$, we have:
\[
(\tilde{\forall}^{\st}z)(\exists x)(\forall y\leq^{*}z)\phi(x, y)\di (\exists x)({\forall}^{\st} y)\phi(x, y)
\]
\item Independence of premises $\textsf{IP}^{\omega}_{\tilde{\forall}^{\st}}$: For any internal $\phi$ and any $\Psi$:
\[
[(\tilde{\forall}^{\st}x)\phi(x)\di (\tilde{\exists}^{\st}y)\Psi(y)]\di (\tilde{\exists}^{\st}z)[(\tilde{\forall}^{\st}x)\phi(x)\di (\tilde{\exists}y\leq^{*}z)\Psi(y)]
\]
\item Nonstandard Markov's principle $\textsf{\textup{M}}^{\omega}$: For any internal $\phi, \psi$, we have 
\[
[(\tilde{\forall}^{\st}x)\phi(x)\di \psi] \di   (\tilde{\exists}^{\st}y)[(\forall x\leq^{*}y)\phi(x)\di \psi]
\]
\item Majorizability axiom \textsf{MAJ}$^{\omega}$:   $(\forall^{\st}x)(\exists^{\st}y)(x\leq^{*} y)$   
\end{itemize}
\end{defi}
Other axioms are mentioned in \cite{dinidini}*{\S4}, but these are derivable in $\M$.  
The variables are not specified for $\textsf{I}^{\omega}$ in \cite{dinidini}, and we have chosen the version from $\IST$. 
We have also added `nonstandard' to the description of the axiom $\textsf{M}^{\omega}$ to distinguish it from the semi-constructive axiom $\MP$, known as `Markov's principle', as follows: 
\be\tag{$\MP$}
(\forall f^{1})\big[ \neg\neg\big[ (\exists n)f(n)=0 \big]\di  (\exists n)f(n)=0\big].
\ee
Now, $\textsf{M}^{\omega}$ implies $\MP$ with all quantifiers relative to `st', which explains the name. 

\subsection{Notations in $\M$}
In this section, we introduce notations relating to $\M$.  

\smallskip

First of all, 
we will use the usual notations for rational and real numbers and functions as introduced in \cite{kohlenbach2}*{p.\ 288-289} (and \cite{simpson2}*{I.8.1} for the former).  
\begin{defi}[Real numbers and related notions in $\RCAo$]\label{keepintireal}\rm~
\begin{itemize}
\item Natural numbers correspond to type zero objects, and we use `$n^{0}$' and `$n\in \N$' interchangeably.  Rational numbers are defined as signed quotients of natural numbers, and `$q\in \Q$' and `$<_{\Q}$' have their usual meaning.    
\item Real numbers are coded by fast-converging Cauchy sequences $q_{(\cdot)}:\N\di \Q$, i.e.\  such that $(\forall n^{0}, i^{0})(|q_{n}-q_{n+i})|<_{\Q} \frac{1}{2^{n}})$.  
We use Kohlenbach's `hat function' from \cite{kohlenbach2}*{p.\ 289} to guarantee that every $f^{1}$ defines a real number.  
\item We write `$x\in \R$' to express that $x^{1}:=(q^{1}_{(\cdot)})$ represents a real as in the previous item and write $[x](k):=q_{k}$ for the $k$-th approximation of $x$.    
\item Two reals $x, y$ represented by $q_{(\cdot)}$ and $r_{(\cdot)}$ are \emph{equal}, denoted $x=_{\R}y$, if $(\forall n^{0})(|q_{n}-r_{n}|\leq \frac{1}{2^{n-1}})$. The inequality `$<_{\R}$' is defined similarly.         
\item Functions $F:\R\di \R$ mapping reals to reals are represented by $\Phi^{1\di 1}$ mapping equal reals to equal reals, i.e. $(\forall x, y)(x=_{\R}y\di \Phi(x)=_{\R}\Phi(y))$.
\item Sets of type $\rho$ objects $X^{\rho\di 0}, Y^{\rho\di 0}, \dots$ are given by their characteristic functions $f^{\rho\di 0}_{X}$, i.e.\ $(\forall x^{\rho})[x\in X\asa f_{X}(x)=_{0}1]$, where $f_{X}^{\rho\di 0}\leq_{\rho\di 0}1$.  
\end{itemize}
\end{defi}
\noindent
Secondly, we use the usual extensional notion of equality.  
\begin{rem}[Equality]\label{equ}\rm
The system $\M$ includes equality between natural numbers `$=_{0}$' as a primitive.  Equality `$=_{\tau}$' for type $\tau$-objects $x,y$ is then defined as:
\be\label{aparth}
[x=_{\tau}y] \equiv (\forall z_{1}^{\tau_{1}}\dots z_{k}^{\tau_{k}})[xz_{1}\dots z_{k}=_{0}yz_{1}\dots z_{k}]
\ee
if the type $\tau$ is composed as $\tau\equiv(\tau_{1}\di \dots\di \tau_{k}\di 0)$.  The inequality `$\leq_{\tau}$' is just \eqref{aparth} with `$\leq_{0}$', i.e.\ binary sequences are given by $f\leq_{1}1$, which we also denote as `$f\in C$' or `$f\in 2^{\N}$'.  
We define `approximate equality $\approx_{\tau}$' as follows:
\be\label{aparth2}
[x\approx_{\tau}y] \equiv (\forall^{\st} z_{1}^{\tau_{1}}\dots z_{k}^{\tau_{k}})[xz_{1}\dots z_{k}=_{0}yz_{1}\dots z_{k}]
\ee
with the type $\tau$ as above.  
The system $\M$ includes the \emph{axiom of extensionality}: 
\be\label{EXT}\tag{$\textsf{\textup{E}}_{\rho\di \tau}$}  
(\forall \varphi^{\rho\di \tau})(\forall  x^{\rho},y^{\rho}) \big[x=_{\rho} y \di \varphi(x)=_{\tau}\varphi(y)   \big].
\ee
for all finite types.  We write $(\textsf{E})$ for the collection of all axioms \eqref{EXT}.
\end{rem}
Finally, we introduce some notation to handle finite sequences nicely.  
\begin{nota}[Finite sequences]\label{skim}\rm
We assume the usual coding of finite sequences of objects of the same type.  
We denote by `$|s|=n$' the length of the finite sequence $s=\langle s_{0}^{\rho},s_{1}^{\rho},\dots,s_{n-1}^{\rho}\rangle$, where $|\langle\rangle|=0$, i.e.\ the empty sequence has length zero.
For sequences $s, t$ of the same type, we denote by `$s*t$' the concatenation of $s$ and $t$, i.e.\ $(s*t)(i)=s(i)$ for $i<|s|$ and $(s*t)(j)=t(|s|-j)$ for $|s|\leq j< |s|+|t|$. For a finite sequence $s$, we define $\overline{s}N:=\langle s(0), s(1), \dots,  s(N-1)\rangle $ for $N^{0}<|s|$.  
For a sequence $\alpha^{0\di \rho}$, we also write $\overline{\alpha}N=\langle \alpha(0), \alpha(1),\dots, \alpha(N-1)\rangle$ for \emph{any} $N^{0}$.  
\end{nota}
Finally, we discuss our use of the term `intuitionistic'.  
\begin{rem}\label{jezus}\rm
As already pointed out in Footnote \ref{haha}, we interpret `intuitionistic' to mean `part of Brouwer's intuitionistic mathematics', due to the first keyword of \cite{dinidini} being `intuitionism'.  
This term also has a more loose interpretation, not uncommon in functional interpretations, meaning that the system is based on intuitionistic logic, i.e.\ does not include all of classical logic. 
These kind of systems are sometimes more correctly called \emph{semi-intuitionistic}.  Due to the topic of \cite{dinidini}, the authors seem to have had the second meaning in mind.     
Nonetheless, our results below show that both senses of the term `intuitionistic' are not completely compatible with $\M$, i.e.\ there is good reason to claim that $\M$ is `merely' non-classical. 
\end{rem}
I thank the referee for pointing out the content of the previous remark.  
\section{Main results}\label{main}
\noindent
Our main results fall into three main categories, as follows.
\begin{enumerate}
\item[(i)] In answer to (Q1), we show in Section \ref{scat} that $\M$ is inconsistent with certain (very) weak fragments of \emph{Transfer}, but (oddly) not with others.  
\item[(ii)] In answer to (Q2), we show in Section \ref{scatto} that $\M$ is inconsistent with certain intuitionistic axioms, relative to the standard world; we also show that $\M$ does prove \emph{weak K\"onig's lemma}, relative to the standard world.  
\item[(iii)] Inspired by these answers to (Q1) and (Q2), we show in Section \ref{nsansa} that $\M$ involves a highly elementary \emph{standard part map}.
\end{enumerate}
In light of the first two items, it seems that $\M$ is not really a system of \emph{intuitionistic} arithmetic (but non-classical nonetheless), while the third item shows that $\M$ already pushes the
boundary of what still counts as `Nonstandard Analysis'.    

\subsection{Non-classical aspects of the Dinis-Gaspar system}\label{scat}
We provide a partial answer to question (Q1) from Section \ref{introke} by showing that $\M$ is inconsistent with various weak fragments of \emph{Transfer}, including \emph{parameter-free Transfer} from \cite{bennosam}, and the Turing jump functional $\exists^{2}$ from e.g.\ \cite{kohlenbach2}, relative to the standard world. 
\subsubsection{Parameter-free Transfer}\label{firstlasteternity}
We show that various extensions of $\M$, also involving intuitionistic axioms, are inconsistent with \emph{parameter-free Transfer} as follows.  
\begin{princ}[$\PFTPE$]
For internal $\varphi(\underline{x})$ with all free variables shown, we have
\be\label{trakke}
(\exists \underline{x})\varphi(\underline{x})\di (\exists^{\st} \underline{x})\varphi(\underline{x}).
\ee
\end{princ}
\noindent
To be absolutely clear, (standard) parameters are \textbf{not} allowed in $\varphi(\underline{x})$ as in \eqref{trakke}.    

\smallskip

In contrast to richer fragments of \emph{Transfer}, $\PFTPE$ is weak: when added to (fragments of) the classical system from \cite{brie}, one obtains a conservative extension, by \cite{bennosam}*{\S3.2}.
The results in \cite{bennosam, samcie18} establish that $\PFTPE$ yields a smooth development of the (classical) \emph{Reverse Mathematics} of Nonstandard Analysis.  

\smallskip

We point out that certain fragments of the axiom of choice (including $\QFAC^{2,0}$ as in the next theorem) are widely accepted in constructive and intuitionistic mathematics (see e.g.\ \cite{beeson1}).   
We also recall \emph{Markov's principle} $\MP$ introduced after Definition \ref{nba} and note that
$\MP$ is rejected in intuitionistic mathematics (\cite{troeleke1}*{p.\ 237}).  
\begin{thm}\label{imkens}
The system $\M+\PFTPE+\QFAC^{2,0}+\textsf{\textup{MP}}$ is inconsistent.  
\end{thm}
\begin{proof}
Recall that $\M$ includes the axiom of extensionality $(\textsf{E}_{2})$, which implies 
\[
(\forall Y^{2}, f^{1}, g^{1})(\exists N^{0})[ \overline{f}N=\overline{g}N\di Y(f)=Y(g)], 
\]
by Markov's principle $\MP$.  Applying $\QFAC^{2,0}$, we obtain $\Phi_{0}^{2\di 0}$ such that 
\[
(\forall Y^{2}, f^{1}, g^{1})(\exists N^{0}\leq \Phi_{0}(Y, f, g))[ \overline{f}N=\overline{g}N\di Y(f)=Y(g)]. 
\]
Applying $\PFTPE$, there is standard such $\Phi_{0}$, yielding that 
\be\label{miethaan}
(\forall^{\st} Y^{2}, f^{1}, g^{1})[f\approx_{1} g\di Y(f)=Y(g)], 
\ee
since $\Phi_{0}(Y, f, g)$ is standard for standard inputs.  Note that \eqref{miethaan} is $(\textsf{E}_{2})^{\st}$, i.e.\ the axiom of `standard extensionality'.   
Now consider the functional $Y_{0}^{2}$ defined as:
\be\label{toxxx}
Y_{0}(f):=
\begin{cases}
0 & (\exists n\leq N+1 )(f(n)= 0) \\
1 & \textup{otherwise}
\end{cases},
\ee
where $N^{0}$ is nonstandard.  Since $Y^{2}_{0}\leq_{2}^{*}1$, and the constant-one-mapping of type two is standard, $Y_{0}$ is also standard by item (b) in the \emph{standardness axioms}.  
Hence, $Y_{0}$ satisfies \eqref{miethaan} and now consider $f_{0}:=11\dots$ and $g_{0}:=\overline{f_{0}}N*00\dots$, which satisfy $f_{0}\approx_{1}g_{0}$ and $Y(f_{0})=0\ne 1=Y(g_{0})$.  
Note that $g_{0}$ is standard by the aforementioned item (b), as $g_{0}\leq_{1}^{*}1$ and the constant-one-mapping of type one is standard, i.e.\ \eqref{miethaan} yields a contradiction. 
\end{proof}
Next, we show that the previous proof also goes through using a \emph{fragment} of Markov's principle $\MP$, called \emph{weak Markov's principle} ($\WMP$ for short; see \cite{ishi1}).  
Most importantly for us, $\WMP$ \emph{is} accepted in intuitionistic mathematics (but not in Bishop's constructive mathematics by the results in \cite{kookje}). 
We will actually use the following\footnote{Note that $\SE$ is actually \emph{weaker} than $\WMP$, but $\SE\asa\WMP$ by \cite{ishibrouwt}*{Thm.~11} in Bishop's constructive mathematics \emph{plus} a non-trivial fragment of the axiom of choice.  
} version of $\WMP$, defined in \cite{ishibrouwt}:   
\be\tag{$\SE$}
(\forall Y^{2}, f^{1}, g^{1})(Y(f)\ne_{0} Y(g)\di f\ne_{1}g ).  
\ee
Since `$x\ne y$' is generally a stronger statement than `$\neg(x=y)$' in constructive mathematics, $\SE$ is said to express \emph{strong extensionality}.  
\begin{cor}
The system $\M+\PFTPE+\QFAC^{2,0}+\SE$ is inconsistent. 
\end{cor}
\begin{proof}
Note that $\SE$ implies the following by considering the \emph{least} such $N$: 
\[
(\forall Y^{2}, f^{1}, g^{1})(\exists N^{0})(Y(f)\ne_{0} Y(g)\di \overline{f}N\ne_{0}\overline{g}N )
\]
As for the theorem, one derives $(\textsf{E}_{2})^{\st}$ and $Y_{0}$ yields a contradiction.  
\end{proof}
The inconsistency in the theorem also pops up when combining $\PFTPE$ with \emph{intuitionistic} axioms, like the \emph{intuitionistic fan functional} (\cite{kohlenbach2, troelstra1}) as follows.
\be\tag{$\MUC$}
(\exists \Omega^{3})(\forall Y^{2})(\forall f, g\leq_{1}1)(\overline{f}\Omega(Y)=\overline{g}\Omega(Y)\di Y(f)=Y(g) )
\ee
\begin{cor}
The system $\M+\PFTPE+\MUC$ is inconsistent.  
\end{cor}
\begin{proof}
Since $\MUC$ is a sentence, $\PFTPE$ guarantees the existence of a standard $\Omega^{3}$ as in the former.  
Now consider $Y_{0}, f_{0}, g_{0}$ from the proof of the theorem and note that $\Omega(Y_{0})$ is a standard number.  Hence, since $f_{0}\approx_{1}g_{0}$ by definition, we have $0=Y(f_{0})=Y(g_{0})=1$, a contradiction.  
\end{proof}
We note that the inconsistency of $\M$ with \emph{much stronger} fragments of \emph{Transfer} is proved in \cite{dinidini}*{Theorem 29}.  In particular, \emph{Transfer} for $\Pi_{1}^{0}$-formulas is used in the latter, which readily translates to the \emph{Turing jump functional} $\exists^{2}$ in the systems from \cite{brie}.  
As it happens, we study $\exists^{2}$ in the next section.  

\smallskip

We do not know whether the classical contraposition of \eqref{trakke} also leads to inconsistency, but we now 
show that it implies the fan theorem, as follows.  
\be\tag{$\FAN$}  
 (\forall T^{1}\leq 1)\big[  (\forall \alpha\leq_{1}1)(\exists m^{0})(\overline{\alpha}m\not\in T)\di  (\exists n^{0})(\forall \beta\leq 1)(\overline{\beta}n\not\in T)  \big]
 \ee 
 The variable `$T^{1}$' is reserved for trees, while `$T\leq_{1}1$' means that $T$ is a binary tree. 
\begin{princ}[$\PFTPA$]
For internal $\varphi(\underline{x})$ with all free variables shown, we have
\be\label{trakketoo}
(\forall^{\st} \underline{x})\varphi(\underline{x})\di (\forall \underline{x})\varphi(\underline{x}).
\ee
\end{princ}
\noindent
To be absolutely clear, (standard) parameters are \textbf{not} allowed in $\varphi(\underline{x})$ as in \eqref{trakketoo}.    
\begin{thm}\label{dorfkesn}
The system $\M+\QFAC^{1,0}+\PFTPA$ proves $\FAN$.
\end{thm}
\begin{proof}
We first prove $\FAN^{\st}$.  
If $(\forall^{\st} f\leq_{1}1)(\exists^{\st}n^{0})(\overline{\alpha}n\not \in T)$, then we have $(\forall f\leq_{1}1)(\exists^{\st} n^{0})(\overline{\alpha}n\not \in T)$ since all binary sequences are standard by item (b) of the nonstandard axioms.  
Applying $\textsf{R}^{\omega}$, we obtain $(\exists^{\st}k^{0})(\forall f\leq_{1}1)(\exists n^{0}\leq k)(\overline{\alpha}n\not \in T)$, and $\FAN^{\st}$ follows. 
The latter immediately implies that
\[
 (\forall^{\st} T^{1}\leq 1, G^{2})\big[  (\forall^{\st} \alpha\leq_{1}1)(\exists m^{0}\leq G(\alpha))(\overline{\alpha}m\not\in T)\di  (\exists^{\st} n^{0})(\forall \beta\leq 1)(\overline{\beta}n\not\in T)  \big].
\]
Now drop the `st' predicates inside the square brackets and apply $\PFTPA$.  The resulting formula then yields $\FAN$, thanks to $\QFAC^{1,0}$.
\end{proof}
Note that $\QFAC^{1,0}$ is `innocent' in that it is included in the base theory of higher-order Reverse Mathematics (see \cite{kohlenbach2}).  
Next, we show that $\PFTPA$ leads to the `full' Heine-Borel compactness of the Cantor space (for uncountable covers), as in:
\be\tag{$\HBU_{C}$}
 (\forall G^{2})(\exists \langle\beta_{0}, \dots, \beta_{k} \rangle )(\forall \alpha\leq 1)(\exists i\leq k)(\alpha \in  [\overline{\beta_{i}} G(\beta_{i}))]).
\ee
Intuitively, any functional $G^{2}$ gives rise to the `canonical' cover $\cup_{f\in C}[\overline{f}G(f)]$ of the Cantor space, and $\HBU_{C}$ tells us that the 
latter always has a finite sub-cover.  
\begin{thm}
The system $\M+\PFTPA$ proves $\HBU_{C}$.
\end{thm}
\begin{proof}
Since every binary sequence is standard in $\M$, we have $(\forall \alpha\leq 1)(\exists^{\st}\beta\leq 1)(\alpha\approx_{1} \beta)$, i.e.\ the nonstandard compactness of the Cantor space.  
However, the usual proof that the latter is equivalent to $\HBU_{C}$ (see \cite{bennosam, samcie18}) does not go through in $\M$ due to the weak conclusion of $\textsf{R}^{\omega}$.  
Instead, we prove \eqref{almozt}, which immediately yields $\HBU_{C}$ via $\PFTPA$.
\be\label{almozt}
 (\forall^{\st} G^{2})(\exists \langle\beta_{0}, \dots, \beta_{k} \rangle )(\forall \alpha\leq 1)(\exists i\leq k)(\alpha \in  [\overline{\beta_{i}} G(\beta_{i}))])
\ee
To prove \eqref{almozt}, fix nonstandard $N$ and define $\beta_{i}:=\sigma*00$ where $\sigma$ is the $i$-th binary sequence of length $N$.  
Then $\langle\beta_{0}, \dots, \beta_{2^{N}} \rangle$ is as required for \eqref{almozt}, as every $\beta_{i}$ is standard by item (b) in the nonstandard axioms, and hence $G(\beta_{i})$ is standard.
Indeed, $(\forall \alpha\leq_{1}1)(\exists i\leq 2^{N})(\beta_{i}\approx_{1} \alpha)$, we are done. 
\end{proof}
As to concluding remarks, Benno van den Berg has suggested `$\varphi(x)\equiv (\exists y)(x\leq^{*}y)$' to show that $\PFTPA$ does not lead to a conservative extension of $\M$.   

\smallskip

Secondly, by \cite{dagsam}*{Cor.~6.7} and \cite{dagsamIII}*{Thm.\ 3.3}, the Turing jump functional $\exists^{2}$ from Section \ref{TJF} and $\HBU_{C}$ give rise to $\ATR_{0}$, and it is 
a natural question whether the same holds for the system $\M+\PFTPA+\QFAC^{2,1}+(\exists^{2})$. 

\smallskip
  
Thirdly, the axiom $\HBU_{C}$ is \emph{extremely hard} to prove: by \cite{dagsamIII}*{\S3.1}, $\Pi_{k}^{1}\textsf{-CA}_{0}^{\omega}$ does not prove $\HBU_{C}$ (for all $k$), where the former is $\RCAo$ plus the existence of $S^{2}_{k}$, a functional which decides the truth of $\Sigma_{k}^{1}$-formulas (only involving type one parameters).  Hence, $\M+\PFTPA$ is a rather peculiar system.  

\smallskip

Fourth, since $\M$ proves that all binary sequences are standard, the following fragment of \emph{Transfer}, introduced in \cite{dagsam}, follows trivially:
\be\tag{\textsf{\textup{WT}}}
(\forall^{\st} Y^{2})\big[ (\exists f^{1}\leq_{1}1)(Y(f)=0)\di(\exists^{\st} f^{1}\leq_{1}1)(Y(f)=0)  \big]
\ee
Note that $\textsf{WT}$ is quite strong: working in the systems from \cite{brie}, $\textsf{WT}$ gives rise to a functional $\kappa^{3}$ which computes a realiser for $\HBU_{C}$, but not vice versa (see \cite{dagsam}*{Theorem 6.17}).  
In fact, the combination of $\kappa^{3}$ and $\exists^{2}$ (from the next section), gives rise to full second-order arithmetic by \cite{dagsam}*{Rem.\ 6.13}.  

\subsubsection{The Turing jump functional}\label{TJF}
We show that the Dinis-Gaspar system is inconsistent with $(\exists^{2})$ \emph{relative to the standard world}.  The axiom $(\exists^{2})$ is given by:
\be\tag{$\exists^{2}$}
(\exists \varphi^{2})(\forall f^{1})[ (\exists n^{0})(f(n)=0)\asa \varphi(f)=0 ].
\ee
Note that $\paai\di (\exists^{2})^{\st}$ by the proof of Corollary \ref{allcomingbacktome}, where $\paai$ is:
\be\label{XxX}
(\forall^{\st}f^{1})\big[  (\forall^{\st}n)(f(n)=0)\di (\forall n)(f(n)=0)  \big], 
\ee
i.e.\ $(\exists^{2})^{\st}$ follows from a fragment of \emph{Transfer} using $\textup{\textsf{M}}^{\omega}$.
\begin{thm}\label{skimitar}
The system $\M+(\exists^{2})^{\st}$ is inconsistent.  
\end{thm}
\begin{proof}
Define the (standard) functional $Z^{1\di 1}$ as $Z(f)(n)=0$ if $f(n)=0$, and $1$ otherwise.  
Since $Z(f)\leq_{1}^{*}1$, the binary sequence $Z(f)$ is standard \emph{for any input $f$}, due to item (b) of the \emph{standardness axioms}.  
Hence, $(\exists^2)^{\st}$ immediately yields:
\be\label{xihu}
(\exists^{\st} \varphi_{0}^{2})(\forall f^{1})[ (\exists^{\st} n^{0})(f(n)=0)\asa \varphi_{0}(f)=0 ], 
\ee
by taking $\varphi_{0}:= \varphi\circ Z$ for $\varphi$ as in $(\exists^{2})^{\st}$.  Clearly, \eqref{xihu} implies
\[
(\forall f^{1})[ \varphi_{0}(f)=0\di (\exists^{\st} n^{0})(f(n)=0) ], 
\]
and applying $\textsf{IP}^{\omega}_{\tilde{\forall}^{\st}}$ yields
\[
(\forall f^{1})(\exists^{\st}m^{0})[ \varphi_{0}(f)=0\di (\exists n^{0}\leq m)(f(n)=0) ], 
\]
while applying $\textsf{R}^{\omega}$ yields:
\be\label{koppel}
(\exists^{\st}k^{0})(\forall f^{1})(\exists m^{0}\leq k)[ \varphi_{0}(f)=0\di (\exists n^{0}\leq m)(f(n)=0) ].  
\ee
Now let $k_{0}$ be a standard number as in \eqref{koppel} and define $f_{0}^{1}$ as $f_{0}(i)=1$ for $i\leq k_{0}+1$, and $0$ otherwise.  
Clearly, $\varphi_{0}(f_{0})=0$ by \eqref{xihu}, but this contradicts \eqref{koppel}.
\end{proof}
The following corollary also follows from the proof of \cite{dinidini}*{Theorem 29}.  
\begin{cor}\label{allcomingbacktome}
The system $\M+\paai$ is inconsistent.  
\end{cor}
\begin{proof}
Note that $Y_{0}$ as in \eqref{toxxx} is standard, while $\paai$ guarantees that $Y_{0}$ behaves just like $\varphi$ in $(\exists^{2})^{\st}$.
Indeed, $\textsf{M}^{\omega}$ implies $\MP^{\st}$, i.e.\ Markov's principle relative to the standard world, and $\paai$ thus implies
\[
(\forall^{\st}f^{1})\big[  (\exists n)(f(n)=0)\di (\exists^{\st} n)(f(n)=0)  \big], 
\]
which is trivially equivalent to \eqref{XxX} in classical logic. 
\end{proof}
Let $\TJ(f, \varphi)$ be the formula in square brackets in $(\exists^{2})$.  
By \cite{bennosam}*{Theorem 4.4}, we have $\paai\asa [ (\exists^{\st} \varphi^{2})(\forall^{\st}f^{1})\TJ(f, \varphi)+ (\textsf{E}_{2})^{\st}]$ over a weak classical system. 
Surprisingly, \emph{only} the final conjunct gives rise to inconsistency, which is implicit in the proof of Theorems \ref{imkens} and \ref{ravudavu}.  
\begin{cor}\label{hessence}
$\M+ (\exists^{\st} \varphi^{2})(\forall^{\st}f^{1})\TJ(f, \varphi)$ is consistent if $\textsf{\textup{E-HA}}^{\omega}+(\exists^{2})$ is.  
\end{cor}
\begin{proof}
Using the same trick involving $Z$ as in the theorem, $ (\exists^{\st} \varphi^{2})(\forall^{\st}f^{1})\TJ(f, \varphi)$ is equivalent to $ (\exists^{\st} \varphi^{2})(\forall f^{1})\TJ(f, \varphi)$.  
The latter follows from $(\exists^{2})$ by taking such $\varphi$ and defining $\varphi_{1}(f)=1$ if $\varphi(f)\ne 0$, and $0$ otherwise.  
Since $\varphi_{1}\leq_{2}^{*}1$, this functional is standard by item (b) of the standardness axioms.  As $(\exists^{2})$ is internal, the theorem now follows from 
the soundness theorem as in \cite{dinidini}*{Theorem 16}, since $\textsf{E-HA}_{\st}^{\omega}$ and $\textsf{E-HA}^{\omega}$ prove the same internal formulas.  
\end{proof}
\begin{cor}\label{skimitar2}
The system $\M+(\textsf{\textup{E}}_{2})^{\st}$ is inconsistent.  
\end{cor}
It should be noted that $\M$ is even inconsistent with the \emph{rule} version of the axiom $(\textsf{\textup{E}}_{2})^{\st}$.  
Indeed, $\M$ proves that $Y_{0}, f_{0}, g_{0}$ from the proof of Theorem \ref{imkens} are standard and satisfy $f_{0}\approx_{1} g_{0}$.  
However, a proof of $Y_{0}(f_{0})=Y_{0}(g_{0})$, say obtained by the aforementioned rule, then leads to a contradiction.  
\begin{cor}\label{Vasjdortrack}
The system $\M+\QFAC^{1,0}+(\exists^{2})+\PFTPE$ is inconsistent.  
\end{cor}
\begin{proof}
Using $\QFAC^{1,0}$, $(\exists^{2})$ readily implies
\[
(\exists \varphi^{2}, \Psi^{2})(\forall f^{1})\big[ ((\exists n^{0})(f(n)=0)\di \varphi(f)=0) \wedge( \varphi(f)=0\di f(\Psi(f))=0) \big], 
\]
where $\Psi(f)$ is the least such $n$ if existent.  By $\PFTPE$, there is standard such $\Psi^{2}$, upon which we obtain $\paai$, a contradiction by Corollary \ref{allcomingbacktome}. 
\end{proof}
In conclusion, while $\M$ is inconsistent with a number of fragments of \emph{Transfer}, the inconsistency is really due to the axiom of extensionality relative to the standard world, 
and not e.g.\ the Turing jump functional as in Corollary \ref{hessence}.     
Since the axiom of extensionality is not rejected in constructive (esp.\ intuitionistic) mathematics, all we can say is that these results suggest that $\M$ is \emph{non-classical}.

\smallskip

Furthermore, $\M$ proves $\neg(\exists^{2})^{\st}$ by Theorem \ref{skimitar}, and classically $\neg(\exists^{2})$ is equivalent to the continuity of all functionals on the Baire space (\cite{kohlenbach2}*{Prop.\ 3.7}); a similar equivalence involving $\SE$ holds constructively by \cite{ishibrouwt}*{Thm.\ 26}.  
However, these equivalences use \emph{Grilliot's trick} (see \cite{kohlenbach2}*{p.\ 287}) and hence require the axiom of extensionality (in some form or other).  Thus, to derive (intuitionistic) continuity theorems from $\neg(\exists^{2})^{\st}$, one would need \emph{standard extensionality}, which leads to inconsistency by Corollary \ref{skimitar2}.  Thus, $\M$ is definitely \emph{non-classical}, but not really intuitionistic.     

\subsubsection{Arithmetical comprehension}\label{ACAS}
We show that the Dinis-Gaspar system is inconsistent with $\ACA_{0}$ \emph{relative to the standard world}.  The axiom $\ACA_{0}$ is:
\be\tag{$\ACA_{0}$}
(\forall f\leq 1)(\exists g\leq 1)(\forall n^{0})[ (\exists m)(f(n,m)=0) \asa g(n)=0  ].
\ee
Our formulation of arithmetical comprehension as in $\ACA_{0}$ makes use of functions, while the version used in RM (see \cite{simpson2}*{II}) makes use of sets. 
These versions are equivalent in light of \cite{simpson2}*{II.3}. 
We single out $\ACA_{0}$ lest anyone believe the inconsistency in Theorem \ref{skimitar} is due to the presence of third-order objects.  
\begin{thm}\label{corfefet}
The system $\M+\ACA_{0}^{\st}$ is inconsistent. 
\end{thm}
\begin{proof}
Since all binary sequences are standard in $\M$, $\ACA_{0}^{\st}$ implies that for all $f\leq1$, there is standard $g\leq1$ such that
\[
(\forall^{\st} n^{0})[ (\exists^{\st} m)(f(n,m)=0) \di  g(n)=0]  \wedge (\forall^{\st}k)[ g(k)=0 \di (\exists^{\st} l)(f(k,l)=0) ].
\]
The second conjunct yields $(\exists^{\st}h^{1})(\forall^{\st}k)[ g(k)=0 \di (\exists l\leq h(k))(f(k,l)=0)$ due to $\textsf{IP}_{\tilde{\forall}^{\st}}^{\omega}$ and $\textsf{mAC}^{\omega}$.  
Thus,  $(\forall f\leq 1)(\exists^{\st}h)(\exists^{\st}g\leq 1)A(f, g, h)$, where $A(f, g,h )$ is 
\[
(\forall^{\st} n^{0})[ (\exists^{\st} m)(f(n,m)=0) \di  g(n)=0]  \wedge (\forall^{\st}k)[ g(k)=0 \di (\exists l\leq h(k))(f(k,l)=0) ].
\]
Since \emph{realisation} $\textsf{R}^{\omega}$ also applies to external formulas, we obtain
\be\label{opil}
(\exists^{\st} h_{0}^{1})(\forall f\leq 1)(\exists h\leq_{1}^{*}h_{0})(\exists^{\st}g\leq 1)A(f, g, h)
\ee
Now define $f_{0}(n,m)$ as $0$ if $m>h_{0}(n)$, and $1$ otherwise, where $h_{0}$ is as in \eqref{opil}.  
For this $f_{0}$, \eqref{opil} provides $g_{0}$, which satisfies by definition:
\[
(\forall^{\st}n)\big[ (\exists^{\st} m)(f_{0}(n,m)=0) \di ( g_{0}(n)=0) \di  (\exists  m\leq h_{0}(n))(f_{0}(n,m)=0) \big],
\]
which contradicts the definition of $f_{0}$, and we are done.  
\end{proof}
It is tempting, \emph{but incorrect}, to apply the reasoning from the previous proof to 
\be\label{inco}
(\forall f\leq 1)[(\exists^{\st}n)(f(n)=0)\di \underline{(\exists^{\st}m)}(f(m)=0)].  
\ee
Indeed, $\textsf{IP}_{\tilde{\forall}^{\st}}^{\omega}$ does \emph{not} allow pulling the underlined
quantifier in \eqref{inco} to the front.  

\smallskip

Finally, while $\M$ proves the non-classical $\neg (\ACA_{0}^{\st})$, we show in Section \ref{UNT!} that it does prove the classical $\WKL^{\st}$, i.e.\ the latter 
does not lead to inconsistency. 

\subsubsection{Non-classical continuity}
We show that relative to the standard world, extensional functions on the Cantor space are automatically continuous on $C$.  
We also show that they are \emph{nonstandard} continuous as follows:
\be\label{gafot}
(\forall^{\st} f\in C)(\forall g\in C)(f\approx_{1}g \di Y(f)=Y(g)).
\ee
Using $\textup{\textsf{M}}^{\omega}$ and $\textsf{R}^{\omega}$, one readily shows that \eqref{gafot} implies `epsilon-delta' continuity relative to the standard world, and the latter implies \eqref{gafot} using item (b) of the nonstandard axioms of $\M$.
Note that \emph{uniform} nonstandard continuity is \eqref{gafot} with the leading `st' dropped. 
\begin{thm}\label{kafir}
The system $\M$ proves that any $Y^{2}$ satisfying $(\textsf{\textup{E}}_{2})^{\st}$ is also nonstandard \(uniformly\) continuous on the Cantor space. 
\end{thm}
\begin{proof}
Suppose $Y^{2}$ satisfies $(\textsf{\textup{E}}_{2})^{\st}$, which immediately yields:
\be\label{gafot2}
(\forall f, g\in C)(\exists^{\st}N^{0})(\overline{f}N=\overline{g}N\di Y(f)=Y(g)). 
\ee
using $\textup{\textsf{M}}^{\omega}$ and the fact that all binary sequences are standard in $\M$.  Applying $\textsf{R}^{\omega}$ to \eqref{gafot2} yields that $Y^{2}$ is nonstandard (uniformly) continuous.  
\end{proof}
Note that, by the proof Theorem \ref{imkens}, there are plenty (standard) functionals $Y^{2}$ that are \emph{not} standard extensional as in $(\textsf{\textup{E}}_{2})^{\st}$. 

\smallskip

Theorem \ref{kafir} can be interpreted as saying that $\M$ has intuitionistic features (in that `more' functionals are continuous than in classical mathematics), but the following corollary shows that something `much more non-classical' is going on.  
A functional $Y^{2}$ is \emph{near-standard} if $(\forall^{\st}f^{1})(\exists^{\st}n)(Y(f)=n)$, as defined in \cite{robinson1}*{p.\ 93}. 
\begin{cor}\label{frenemey}
The system $\M$ proves that for any near-standard $Y^{2}$ satisfying $(\textsf{\textup{E}}_{2})^{\st}$, there is standard $Z^{2}$ such that $(\forall f\in C)(Z(f)=_{0}Y(f))$.  
\end{cor}
\begin{proof}
First of all, since all binary sequences are standard, $(\forall f\in C)(\exists^{\st}n)(Y(f)=n)$ follows from the near-standardness of $Y$, and applying $\textsf{R}^{\omega}$ yields a standard upper bound $n_{0}$ for $Y^{2}$ on the Cantor space.  
Fix nonstandard $N_{0}$ and define $Z(f)$ as $Y(\overline{f}N_{0}*00\dots )$ if $\overline{f}N_{0}$ is a binary sequence, and $n_{0}$ otherwise.  Then $Z(f)=Y(f)$ for $f\in C$ by standard extensionality, and $Z\leq_{2}^{*}n_{0}$ implies that $Z$ is standard.    
\end{proof}
By the theorem, standard extensionality implies continuity relative to the standard world.  Now, as discussed in \cites{brie, SB}, one can naturally interpret the standardness predicate `$\st(x)$' as `$x$ is computationally relevant' using the  systems from \cite{brie}.  
With this interpretation in mind, Corollary \ref{frenemey} expresses that relative to `st', continuity implies being computable (in some sense).  
However, intuitionistic mathematics, the continuity axiom $\WCN$ in particular, refutes \emph{Church's thesis} \textsf{CT}, where the latter expresses that all sequences are computable (in the sense of Turing), and the former
implies Brouwer's continuity theorem (see \cite{troeleke1}*{p.\ 211}).  

\smallskip

We can even prove a stronger consequence of Theorem \ref{kafir}, as follows. 
\begin{cor}\label{frenemey2}
The system $\M+(\exists^{2})$ proves that for any near-standard $Y^{2}$ and standard $g^{1}$, there is standard $Z^{2}$ such that $(\forall f\leq_{1} g)(Z(f)=Y(f))$.  
\end{cor}
\begin{proof}
Use $(\exists^{2})$ to define $Z(f)$ as $Y(f )$ if $f\leq_{1} g$ and $0$ otherwise.  Then $Z$ is standard in the same way as in the corollary: since $g$ is standard, $f\leq_{1}g$ is too.
\end{proof}
By the previous, any functional $Y^{2}$ is automatically standard if it is near-standard on $C$, and zero elsewhere.  

\subsection{Non-intuitionistic aspects of the Dinis-Gaspar system}\label{scatto}
We show that the Dinis-Gaspar system does not qualify as a system of intuitionistic mathematics for the following reasons:
\begin{enumerate}
\item[(i)] The system $\M$ proves, relative to the standard world, the \emph{weak K\"onig's lemma}, which is rejected in constructive mathematics (Section \ref{UNT!}). 
\item[(ii)] The system $\M$ is inconsistent with the axiom, relative to the standard world, \emph{all functions are \(epsilon-delta\) continuous on the Baire space} (Section~\ref{hongry}).
\item[(iii)] The system $\M$ is inconsistent with the axiom schema, relative to the standard world, called \emph{Kripke's scheme} (Section \ref{hongry2}).
\end{enumerate}
Regarding the occurrence of `relative to the standard world' in the previous items, we recall the following regarding the standard objects in internal set theory.  
\begin{quote}
For example, the set $\N$ of all natural numbers, the set $\R$ of all real numbers, the real number $\pi$, and the Hilbert space $L^{2}(\R)$ are all standard sets, since they may be uniquely described in conventional mathematical terms.
\emph{Every specific object of conventional mathematics is a standard set}. It remains unchanged in the new theory. (\cite{wownelly}*{p.\ 1166}, emphasis in original)
\end{quote}
We note that all closed terms of $\M$ are standard, and presumably every object which may be constructed (in some sense or other from constructive mathematics) will be standard.  
Moreover, even in the classical system from \cite{brie}, the \emph{standard} objects yield (copious) computational/constructive content, as detailed in \cite{SB}.  
Thus, the standard world should be the focus of our attention, if we are interested in computational/constructive content.  
\subsubsection{Weak K\"onig's lemma}\label{UNT!}
We show that the Dinis-Gaspar system proves, relative to the standard world, \emph{weak K\"onig's lemma} and the latter's uniform version.
Recall that the variable `$T$' is reserved for trees, and we denote by `$T\leq_{1}1$' that $T$ is a binary tree.  Then $\WKL$ is just the classical contraposition of $\FAN$, and   
\be\tag{$\UWKL$}
(\exists \Psi) (\forall T^{1}\leq 1)\big[   (\forall n^{0})(\exists \beta\leq 1)(\overline{\beta}n\in T)  \di(\forall m^{0})(\overline{\Psi(T)}m\in T) \big]
\ee
the uniform version. 
As $\WKL$ is (constructively) equivalent to a fragment of the law of excluded middle (see \cite{ishi1}), it is rejected in constructive mathematics.  
\begin{thm}
The system $\M$ proves $\WKL^{\st}$ and $\UWKL^{\st}$.  
\end{thm}
\begin{proof}
Let $f_{0}^{1}$ be the sequence that is constant $0$.  
Let $T$ be a standard binary tree such that $(\forall^{\st}n^{0})(\exists \sigma^{1}\leq 1)(\overline{\sigma}n\in T)$, i.e.\ $T$ is infinite relative to the standard world.  
We immediately obtain:
\[
(\forall^{\st}n^{0})(\exists \sigma^{1}\leq 1)(\forall m\leq n)(\overline{\sigma}m\in T),
\]
and applying $\textsf{I}^{\omega}$ (since `$\leq_{0}^{*}$' is `$\leq_{0}$' by definition) yields $(\exists \sigma^{1}\leq 1)(\forall^{\st}n)(\overline{\sigma}n\in T)$. 
Since $\sigma \leq_{1}^{*}1$, item (b) of the nonstandard axioms implies that $\sigma$ is a standard binary sequence, and $\WKL^{\st}$ follows.  To obtain $\UWKL^{\st}$, fix nonstandard $N$ and define $\Phi^{1\di 1}$ as follows: 
$\Phi(T)$ is $\sigma*f_{0}$ where $\sigma\in T$ is the left-most binary sequence of maximal length $|\sigma|\leq N $, if it exists, and $f_{0}$ otherwise.  Since $\Phi\leq_{1\di 1}^{*}1$, this defines a standard functional, and we are done.   
\end{proof}
Kohlenbach shows in \cite{kooltje} that $\RCAo\vdash \UWKL\asa (\exists^{2})$ \emph{crucially} depends on the axiom of extensionality.  
Assuming $\M$ is consistent, we do not have access to $(\textsf{E}_{2})^{\st}$ by the proof of Theorem \ref{skimitar}, 
and hence $(\exists^{2})^{\st}$ does not follow from $\UWKL^{\st}$ in $\M$, i.e.\ the previous theorem does not lead to a contradiction. 
Moreover, the first part of the theorem, involving a classical system, has been proved in \cite{dife}, and the proof in the latter seems to go through in our (semi-)intuitionistic setting. 
 
\smallskip

Moreover, $\WKL$ is (constructively) equivalent to $(\forall x\in \R)(x\geq 0 \vee x\leq 0)$ and to the fact that every real in $[0,1]$ has a binary representation (see \cite{ishi1}).  
As expected, $\M$ also proves versions of the latter, relative to the standard world.   
\begin{thm}\label{dorkiiii}
The system $\M$ proves that every real in the unit interval has a standard binary approximation, i.e.\ $(\forall x\in [0,1])(\exists^{\st}f\in C)\big(x\approx \sum_{n=0}^{\infty} \frac{f(n)}{2^{n}}\big)$, and  
\be\label{simpler}
(\exists^{\st} \Phi^{2})(\forall x\in \R)(\Phi(x)=0\di x\lessapprox 0 \wedge \Phi(x)=1 \di x \gtrapprox 0).
\ee
\end{thm}
\begin{proof}
Fix nonstandard $N$ and define $\Phi^{2}$ as:  $\Phi(x)=0$ if $[x](N)\leq \frac{1}{N}$, and $1$ otherwise.  
Note that $\Phi\leq_{2}^{*}1$ implies this functional is standard.  Then $\Phi(x_{0}-1/2)$ provides the first bit of a binary approximation of $x_{0}$, and given the first $n$ such bits $b_{0}, \dots, b_{n-1}$, then $\Phi(x_{0}-(\frac{1}{2^{n+1}}+\sum_{i=0}^{n-1}\frac{b_{i}}{2^{i+1}}))$ yields the $n+1$-th bit.  
\end{proof}
There are a number of other theorems (constructively) equivalent to $\WKL$ by \cite{ishi1}, like e.g.\ the intermediate value theorem.  
As expected, one can also establish these theorems relative to `st' inside $\M$, but we do not go into details. 

\smallskip

It is well-known that $\WKL$ is inconsistent with the aforementioned axiom \emph{Church's thesis} $\textsf{CT}$ (\cite{beeson1}*{p.\ 68}).  
Since $\M$ proves $\WKL^{\st}$, one expects $\M$ to be inconsistent with $\textsf{CT}$ relative to the standard world.  Let `$\varphi_{e,s}(n)=m$' be the (primitive recursive) predicate expressing that the 
Turing machine with index $e$ and input $n$ halts after at most $s$ steps with output $m$.  Then \emph{Church's thesis} is defined as follows.    
\be\tag{$\textsf{\textup{CT}}$}
(\forall f^{1})(\exists e^{0})(\forall n^{0},m^{0})\big[ (\exists s^{0}) (\varphi_{e, s}(n)=m)\asa f(n)=m  \big].
\ee
\begin{thm}
The system $\M$ proves $\neg \textsf{\textup{CT}}^{\st}$.
\end{thm}
\begin{proof}
Suppose $\textsf{CT}^{\st}$ holds.  
Fix nonstandard $N^{0}$ and define (standard by definition) $f_{0}\leq1$ as follows: $f_{0}(e)=1$ if $(\exists s\leq N)(\varphi_{e,s}(e)=0)$, and $0$ otherwise. 
Then there is standard $e_{0}^{0}$ such that $(\exists^{\st}  s^{0})(  \varphi_{e_{0}, s}(e_{0})=m)\asa f_{0}(e_{0})=m  $ for any standard $m$.  
However, $f_{0}(e_{0})=1$ implies by definition $(\exists s^{0}\leq N)(\varphi_{e_{0},s}(e_{0})=0)$, a contradiction. 
Similarly, $f_{0}(e_{0})=0$ implies by definition $(\forall  s^{0}\leq N)(\forall n^{0})(\varphi_{e_{0},s}(e_{0})=n\di n\ne0)$, a contradiction. 
Since we obtained a contradiction in each case, $\textsf{CT}^{\st}$ is false. 
\end{proof}

\subsubsection{Intuitionistic continuity}\label{hongry}
We show that $\M$ is inconsistent with certain axioms, relativised to the standard world, of intuitionistic mathematics.  

\smallskip

First of all, we consider the continuity principle $\BCT_{C}\equiv(\forall Y^{2})\cont_{C}(Y)$, which expresses that all functionals are (epsilon-delta) continuous on the Cantor space, as given by the following formula:
\be\tag{$\cont_{C}(Y)$}
(\forall   f \leq 1)(\exists N^{0})(\forall g\leq 1)(\overline{f}N=\overline{g}N\di Y(f)=Y(g)).
\ee
Secondly, we consider the principle \emph{weak continuity for numbers}
\be\tag{\textsf{\textup{WC-N}}}
(\forall \alpha^{1})(\exists n^{0})A(\alpha, n)\di (\forall \alpha^{1})(\exists n^{0}, m^{0})(\forall \beta^{1})[ \overline{\alpha}n=\overline{\beta}m\di  A(\alpha, m)]
\ee
for any formula $A$ in the language of finite types.  Let $\WCN_{0}$ be the restriction of $\WCN$ to quantifier-free formulas, and recall the axiom $\SE$ from Section \ref{firstlasteternity}.  
\begin{thm}\label{ravudavu}
The systems $\M+(\BCT_{C})^{\st}$, $\M+(\WCN_{0})^{\st}$, and $\M+\SE^{\st}$ are inconsistent. 
\end{thm}
\begin{proof}
For the first part, consider $Y_{0}, f_{0}, g_{0}$ as in the proof of Theorem \ref{imkens} and note that $f_{0}\approx g_{0}$ contradicts $(\BCT_{C})^{\st}$.
For the second part, take $A(\alpha, n)\equiv( Y_{0}=n)$ and note that $(\WCN_{0})^{\st}$ implies that $Y_{0}$ is epsilon-delta continuous on $C$, relative to the standard world. 
For the third part, note that $\SE^{\st}$ implies $(\textsf{E}_{2})^{\st}$. 
\end{proof}
Note that $Y_{0}$ is not sequentially continuous relative to the standard world, i.e.\ the restriction of $\BCT_{C}$ to sequential continuity does not change the previous theorem. 
Moreover, due to $\textup{\textsf{M}}^{\omega}$, there is no difference between $\textsf{LPO}^{\st}$ and the weaker $\textsf{WLPO}^{\st}$, i.e.\ the associated notion of \emph{nondiscontinuity} (\cite{ishi1}*{Thm.\ 3}) is not relevant here.  

\smallskip

As an aside, $\SE$ follows from $\WMP$ by \cite{ishibrouwt}*{Thm.\ 11}, which in turns is provable in (constructive) recursive mathematics (see \cite{ishi1}*{Prop.\ 13}).
Hence, $\M$ is also inconsistent with theorems of recursive mathematics, relative to the standard world.  

\smallskip

As another aside, we prove that $\M$ is consistent (or even outright proves) certain theorems of intuitionistic mathematics.  
Indeed, a \emph{consequence} of $\BCT_{C}$ (together with $\FAN$) is that all functions on $C$ are bounded.   
\begin{thm}\label{ravudavu2}
The system $\M$ proves $(\forall^{\st}Y^{2})(\exists^{\st} N^{0})(\forall^{\st} f\leq 1)(Y(f)\leq N)$; the system $\M+\PFTPA$ proves $(\forall Y^{2})(\exists  N^{0})(\forall f\leq 1)(Y(f)\leq N)$.
\end{thm}
\begin{proof}
For standard $Y^{2}$, since all binary sequences are standard, we have $(\forall f\leq 1)(\exists^{\st}n^{0})(Y(f)\leq n)$, and $\textsf{R}^{\omega}$ finishes the first part. 
For the second part, drop all but the leading `st' and apply $\PFTPA$.
\end{proof}
The previous implies that $\M+\PFTPA$ is inconsistent with recursive mathematics, as the latter involves \emph{unbounded} functionals on $2^{\N}$ (see \cite{beeson1}*{p.\ 70}).
In particular, $\M+\PFTPA+\textsf{CT}$ is inconsistent, which also follows from Theorem \ref{dorfkesn} if we in addition add $\QFAC^{1,0}$ to the system. 

\smallskip

Finally, we show that the Dinis-Gaspar system is inconsistent with a \emph{classical} continuity principle.  Our motivation is to exclude an incorrect interpretation of the results in the previous two sections.  
Indeed, one could say that $\M$ is slightly classical (as it proves $\WKL^{\st}$) and \emph{therefore} Theorem \ref{ravudavu}.  
As it turns out, $\M$ is inconsistent with $(\BCT_{C})^{\st}$ restricted to continuous functionals.    

\smallskip

Thus, define $\CCT_{C}\equiv(\forall^{\st} Y^{2})(\cont_{C}(Y) \di [\cont_{C}(Y)]^{\st})$, which expresses that all functionals which are (epsilon-delta) continuous on $C$, are also continuous in this way \emph{relative to the standard world}.  Note that $\CCT_{C}$ readily follows from \emph{Transfer}.   
\begin{thm}
The system $\M+\CCT_{C}$ is inconsistent. 
\end{thm}
\begin{proof}
Consider the standard objects $Y_{0}, f_{0}, g_{0}$ as in the proof of Theorem \ref{imkens} and note that $f_{0}\approx g_{0}$ contradicts $\CCT_{C}$ as $\cont_{C}(Y_{0})$.
\end{proof} 
One could replace the antecedent of $\CCT_{C}$ with more restrictive internal formulas, but the end result would still be the same. 

\subsubsection{Kripke's scheme}\label{hongry2}
We show that $\M$ is inconsistent with a fragment of \emph{Kripke's scheme} relative to the standard world.  
This is not \emph{that} surprising since $\M$ includes nonstandard Markov's principle $\textup{\textsf{M}}^{\omega}$, which implies $\MP^{\st}$, i.e.\ Markov's principle $\MP$ relative to the standard world.  
Indeed, Markov's principle $\MP$ is rejected in intuitionistic mathematics, which was first established by Brouwer using an axiom scheme nowadays called {Kripke's scheme} (see \cite{dummy}*{p.\ 244} for details).  
The `strong' form of this scheme is formulated as follows by Dummett in \cite{dummy}.    
\begin{princ}[$\textsf{KS}^{*}$]
For any formula $A$, we have
\[
(\tilde{\exists} \beta\leq 1)(A \asa (\exists n^{0})(\beta(n)=1)).
\]
\end{princ}
\noindent
We consider the following special case of $\textsf{KS}^{*}$:
\be\tag{$\textsf{\textup{KS}}_{0}^{*}$}
(\forall \alpha\leq 1)(\exists \beta\leq 1)(\forall m^{0})\big[  (\forall k^{0})\alpha(k, m)=0\asa (\exists n)(\beta(n,m)=0)  \big].
\ee
By \cite{troeleke1}*{\S9.5}, Markov's principle $\MP$ and the Kripke schema imply the law of excluded middle, which 
is a similar result to what is obtained in the following proof. 
\begin{thm}
The system $\M+(\textsf{\textup{KS}}_{0}^{*})^{\st}$ is inconsistent. 
\end{thm}
\begin{proof}
Fix nonstandard $N^{0}$ and fix standard $\alpha, \beta\leq 1$ as in $(\textsf{\textup{KS}}_{0}^{*})^{\st}$;  let $g_{0}(m)$ (resp.\ $h_{0}(m)$) be the least $k\leq N$ such that $\alpha(k,m)\ne 0$ (resp.\ $\beta(k,m)=0$) if it exists, and $N$ otherwise.  
Define (standard by definition) $\gamma\leq_{1}1$ such that $\gamma(m)=0$ if $g_{0}(m)>h_{0}(m)$, and $1$ otherwise.  Then if we can prove the following: 
\be\label{corfefe}
(\forall^{\st} m^{0})\big[  (\forall^{\st} k^{0})\alpha(k, m)=0\asa (\exists^{\st} n)(\beta(n,m)=0) \asa \gamma(m)=0\big], 
\ee
then we are done: $\textup{\textsf{M}}^{\omega}$ guarantees that \eqref{corfefe} implies $\ACA_{0}^{\st}$ from Section \ref{ACAS}, and Theorem \ref{corfefet} yields the desired contradiction. 
To prove \eqref{corfefe}, if for standard $m$, we have $(\exists^{\st} n)(\beta(n,m)=0)$, then $h_{0}(m)$ is standard, while $g_{0}(m)$ is nonstandard (by the first equivalence in \eqref{corfefe}), i.e.\ $h_{0}(m)< g_{0}(m)$.  
Note that $(\forall^{\st} k^{0})\alpha(k, m)=0$ implies $(\forall  k^{0}\leq K_{0})\alpha(k, m)=0$ for some nonstandard $K_{0}$ using \emph{Idealisation} $\textsf{I}^{\omega}$ as usual.  The reverse implication follows in 
the same way using $\textup{\textsf{M}}^{\omega}$.  
\end{proof}

\subsection{Non-standard aspects of the Dinis-Gaspar system}\label{nsansa}
We show that the system $\M$ includes a `standard part map', a notion introduced in the next paragraph.
As we will see, this raises the question to what extent $\M$ (and the system from \cite{fega}) can still be referred to as `Nonstandard Analysis' or `internal set theory'. 

\smallskip

First of all, Robinson introduces the `standard part map' $^{\circ}$ in \cite{robinson1}*{p.~57}; the latter maps any $x\in [0,1]$ to the (unique) \emph{standard} $^{\circ}x$ such that $x\approx {^{\circ}x}$, and the latter is called the `standard part' of the former.   
However, in the Robinsonian framework, the standard part map is \emph{external}.

\smallskip

Secondly, in light of the previous, there is no hope of having access to this map in Nelson's $\IST$: we are only given the \emph{Standardisation} axiom in which the standard part of a real \emph{exists}.  
Nonetheless, we show that $\M$ does afford a standard part map, and even a generalisation to functionals on the Cantor space.   

\begin{thm}
There is a term $u^{(1\times 0)\di 1}$ of G\"odel's $T$ such that $\M$ proves: for nonstandard $N^{0}$ and $x\in [0,1]$, we have $ \st^{1}(u(x, N))$ and $u(x, N)\approx x$.
\end{thm}
\begin{proof}
Let $f_{0}^{1}$ be the constant zero function. 
Recall the functional $\Phi$ form Theorem \ref{dorkiiii} and fix nonstandard $N$; define $v(x,N)$ as $\Psi(x,N)*f_{0}$ if $-\frac{1}{N}\leq_{\Q} [x](2^{N})\leq_{\Q}1+\frac{1}{N} $, and $f_{0}$ otherwise.  
Here, $\Psi(x, 0)$ is $\langle \Phi(x-\frac{1}{2})\rangle$ and $\Psi(x, n+1)$ is $\Psi(x, n)*\langle b\rangle$, where $b=\Phi(x -(\frac{1}{2^{n+1}}+\sum_{i=0}^{n-1}\frac{\Phi(x, n)(i)}{2^{i+1}})  )$.
Since $v(x,N)\leq_{1}^{*}1$, the former is standard (in the sense that $\st^{1}(v(x,N))$ for any $x\in [0,1]$), and satisfies $ \sum_{n=0}^{\infty}\frac{v(x, N)(n)}{2^{n+1}}\approx x$ by design.  Define standard $w^{1\di1}$ as $w(\alpha)(n):=\sum_{i=0}^{n}\frac{\alpha(n)}{2^{n}}$, and note that $u:=w\circ v$ is as required by the theorem. 
\end{proof}
Recall that we (may) view any sequence as a real; since $\lambda x.v(x, N)\leq_{1\di 1}^{*}1$ we have $\st^{1\di 1}( \lambda x.v(x, N))$, and the standard part map $u:=w \circ v$ is thus \emph{standard} in $\M$, a fairly `non-standard' situation as discussed in Remark \ref{blasphemy}.   
\begin{thm}
There is $s^{(2\times 0)\di 2}$ in G\"odel's $T$ such that $\M$ proves: for nonstandard $N^{0}$ and near-standard $Y^{2}$ such that $(\textsf{\textup{E}}_{2})^{\st}$, we have $ \st^{2}(s(Y, N))\wedge (\forall f\in C)(s(Y, N)(f)=Y(f))$.
\end{thm}
\begin{proof}
By the near-standardness of $Y^{2}$, and the fact that all binary sequences are standard, we have $(\forall f\in C)(\exists^{\st}n)(Y(f)\leq n)$, and $\textsf{R}^{\omega}$ implies 
$(\forall f\in C)(\exists n\leq n_{0})(Y(f)\leq n)$ for some standard $n_{0}$.
Fix nonstandard $N_{0}$ and define $s(Y,N_{0})(f)$ as $Y(\overline{f}N_{0}*00\dots )$ if $\overline{f}N_{0}$ is a binary sequence, and $n_{0}$ otherwise.  Then $s(Y, N_{0})(f)=Y(f)$ for $f\in C$ by standard extensionality, 
and $\lambda f.s(Y, N_{0})(f )\leq_{2}^{*}n_{0}$ implies that $\st^{2}(\lambda f.s(Y, N_{0})(f))$, as required. 
\end{proof}
\begin{cor}
The system $\M+(\exists^{2})$ proves that there is $\Phi^{2\di 2}$ such that for near-standard $Y^{2}$, we have $ \st^{2}(\Phi(Y))\wedge (\forall f\in C)(\Phi(Y)(f)=Y(f))$.
\end{cor}
\begin{proof}
Use $\exists^{2}$ to define $\Phi(Y)(f)$ as $Y(f)$ if $f\in C$, and zero otherwise.  Then $\Phi(Y)$ is standard in the same way as in the theorem.  
\end{proof}
The previous theorem could be obtained for $F:[0,1]\di \R$ using Theorem \ref{dorkiiii}, but this development would mostly be repetitive. 
We finish this section with an \emph{informal} remark on just how unnatural the standard part maps of $\M$ are.  
\begin{rem}\label{blasphemy}\rm
The standard part maps of $\M$ are quite unnatural \emph{from the point of view of internal set theory} for the following reason: the standard part of a real $x\in [0,1]$ is \emph{unique} in $\IST$, i.e.\ if $x\approx y\approx z$ and the latter two are standard reals, then $y=z$. 
Hence, if there were $\Phi:\R\di \R$ such that $\Phi(x)\approx x\wedge \st(\Phi(x))$ for any $x\in [0,1]$, then we observe that $(\forall x\in [0,1])(\st(x)\asa x=\Phi(x)) $.  
However, one of the central tenets of $\IST$ is that `st' is not definable via an internal formula:
\begin{quote}
To assert that $x$ is a standard set has no meaning within conventional mathematics-it is a new undefined notion. (\cite{wownelly}*{p.\ 1165})
\end{quote}
These observations do not cause problems for $\M$ of course: the uniqueness of standard parts in $\IST$ requires \emph{Transfer} anyway, while `$x=_{\R} y\wedge \st(x)$' does not imply $\st(y)$ in $\M$ due to issues of representation of reals. 
Nonetheless, $\M$ is only one basic step removed from being able to define `$\st^{1}$' via an internal formula, something which goes against the very nature of $\IST$.  Although the frameworks are of course different, a similar case can be made for the 
Robinsonian approach.   
\end{rem}
Now, the law of excluded middle is referred to as a `taboo' in constructive mathematics (see \cite{beeson1}*{I.3}).  
In light of the previous remark, those endorsing this kind of language should probably use \emph{heresy} when referring to the above standard part maps of $\M$ in the context of Nonstandard Analysis and internal set theory.  

\section{Conclusion}\label{konkelfoes}
In the previous sections, we have provided fairly conclusive answers to questions (Q1) and (Q2) from Section \ref{introke}. 
We isolated (very) weak fragments of \emph{Transfer} which are still inconsistent with $\M$, and we identified a number of 
axioms of intuitionistic (and general constructive) mathematics which are inconsistent with $\M$ when formulated relative to the standard world.   
We even established that $\M$ allows for a highly elementary \emph{standard part map}, a rather `non-standard' feature of $\M$.   

\smallskip

These facts all suggest -in one way or another- that $\M$ is indeed \emph{non-classical}, but does not really deserve the description \emph{intuitionistic}.  
At the same time, since a standard part map is not available in Nelson's internal set theory, and \emph{external} in Robinson's approach, $\M$ really pushes the boundary of what still counts as `Nonstandard Analysis' and `internal set theory'.  

\smallskip

In our opinion, the aforementioned problems trace back to \emph{one} problematic axiom of $\M$, namely item (b) of the nonstandard axioms.  
Simply put, this axiom `makes too many things standard', an obvious example being the Cantor space.      
While this axiom may be necessary and/or useful for the connection to the bounded functional interpretation (see \cite{dinidini}*{\S6} and \cite{fega}*{\S4}), 
it is not natural from the point of view of Nonstandard Analysis.

\begin{ack}\rm
This research was supported by the following funding bodies: FWO Flanders, the John Templeton Foundation, the Alexander von Humboldt Foundation, and LMU Munch (via the \emph{Excellence Initiative} and CAS LMU).  
The author expresses his gratitude towards these institutions.   I also thank the referee for the many helpful suggestions.  
\end{ack}


\begin{bibdiv}
\begin{biblist}
\bib{beeson1}{book}{
  author={Beeson, Michael J.},
  title={Foundations of constructive mathematics},
  series={Ergebnisse der Mathematik und ihrer Grenzgebiete},
  volume={6},
  note={Metamathematical studies},
  publisher={Springer},
  date={1985},
  pages={xxiii+466},
}

\bib{brie}{article}{
  author={van den Berg, Benno},
  author={Briseid, Eyvind},
  author={Safarik, Pavol},
  title={A functional interpretation for nonstandard arithmetic},
  journal={Ann. Pure Appl. Logic},
  volume={163},
  date={2012},
  number={12},
  pages={1962--1994},
}

\bib{bennosam}{article}{
  author={van den Berg, Benno},
  author={Sanders, Sam},
  title={Reverse Mathematics and parameter-free Transfer},
  journal={To appear in Annals of Pure and Applied Logic},
  volume={},
  doi={10.1016/j.apal.2018.10.003},
  date={2018},
  number={},
  note={Available on arXiv: \url {http://arxiv.org/abs/1409.6881}},
  pages={},
}

\bib{dife}{article}{
   author={Dinis, Bruno},
   author={Ferreira, Fernando},
   title={Interpreting weak K\"onig's lemma in theories of nonstandard
   arithmetic},
   journal={MLQ Math. Log. Q.},
   volume={63},
   date={2017},
   number={1-2},
   pages={114--123},
}
	

\bib{dinidini}{article}{
  author={Bruno {Dinis} and Jaime {Gaspar}},
  title={{Intuitionistic nonstandard bounded modified realisability and functional interpretation.}},
  journal={{Ann. Pure Appl. Logic}},
  volume={169},
  number={5},
  pages={392--412},
  year={2018},
}



\bib{dummy}{book}{
  author={Michael {Dummett}},
  title={{Elements of intuitionism. 2nd ed.}},
  pages={xii + 331},
  year={2000},
  publisher={Oxford: Clarendon Press},
}

\bib{BFI}{article}{
   author={Ferreira, Fernando},
   author={Oliva, Paulo},
   title={Bounded functional interpretation},
   journal={Ann. Pure Appl. Logic},
   volume={135},
   date={2005},
   number={1-3},
   pages={73--112},
}

\bib{BFI2}{article}{
   author={Ferreira, Fernando},
   title={Injecting uniformities into Peano arithmetic},
   journal={Ann. Pure Appl. Logic},
   volume={157},
   date={2009},
   number={2-3},
   pages={122--129},
}
	

\bib{fega}{article}{
  author={Ferreira, Fernando},
  author={Gaspar, Jaime},
  title={Nonstandardness and the bounded functional interpretation},
  journal={Ann. Pure Appl. Logic},
  volume={166},
  date={2015},
  number={6},
  pages={701--712},
}

\bib{benno2}{article}{
  author={Amar {Hadzihasanovic} and Benno {van den Berg}},
  title={{Nonstandard functional interpretations and categorical models.}},
  journal={{Notre Dame J. Formal Logic}},
  volume={58},
  number={3},
  pages={343--380},
  year={2017},
}

\bib{vajuju}{book}{
  author={van Heijenoort, Jean},
  title={From Frege to G\"odel. A source book in mathematical logic, 1879--1931},
  publisher={Harvard University Press},
  place={Cambridge, Mass.},
  date={1967},
  pages={xi+660 pp. (1 plate)},
}

\bib{ishi1}{article}{
  author={Ishihara, Hajime},
  title={Reverse mathematics in Bishop's constructive mathematics},
  year={2006},
  journal={Philosophia Scientiae (Cahier Sp\'ecial)},
  volume={6},
  pages={43-59},
}

\bib{ishibrouwt}{article}{
  author={Ishihara, Hajime},
  title={On Brouwer's continuity principle},
  journal={To appear in \emph {Indagationes Mathematicae}},
  date={2018},
  pages={pp.\ 22},
}

\bib{kookje}{article}{
  author={Kohlenbach, Ulrich},
  title={On weak Markov's principle},
  note={Dagstuhl Seminar on Computability and Complexity in Analysis, 2001},
  journal={MLQ Math. Log. Q.},
  volume={48},
  date={2002},
  number={suppl. 1},
  pages={59--65},
}

\bib{kooltje}{article}{
  author={Kohlenbach, Ulrich},
  title={On uniform weak K\"onig's lemma},
  note={Commemorative Symposium Dedicated to Anne S. Troelstra (Noordwijkerhout, 1999)},
  journal={Ann. Pure Appl. Logic},
  volume={114},
  date={2002},
  number={1-3},
  pages={103--116},
}

\bib{kohlenbach2}{article}{
  author={Kohlenbach, Ulrich},
  title={Higher order reverse mathematics},
  conference={ title={Reverse mathematics 2001}, },
  book={ series={Lect. Notes Log.}, volume={21}, publisher={ASL}, },
  date={2005},
  pages={281--295},
}

\bib{kohlenbach3}{book}{
  author={Kohlenbach, Ulrich},
  title={Applied proof theory: proof interpretations and their use in mathematics},
  series={Springer Monographs in Mathematics},
  publisher={Springer-Verlag},
  place={Berlin},
  date={2008},
  pages={xx+532},
}

\bib{wownelly}{article}{
  author={Nelson, Edward},
  title={Internal set theory: a new approach to nonstandard analysis},
  journal={Bull. Amer. Math. Soc.},
  volume={83},
  date={1977},
  number={6},
  pages={1165--1198},
}

\bib{dagsam}{article}{
  author={Normann, Dag},
  author={Sanders, Sam},
  title={Nonstandard Analysis, Computability Theory, and their connections},
  journal={Submitted, Available from arXiv: \url {https://arxiv.org/abs/1702.06556}},
  date={2017},
}

\bib{dagsamIII}{article}{
  author={Normann, Dag},
  author={Sanders, Sam},
  title={On the mathematical and foundational significance of the uncountable},
  journal={Submitted, arXiv: \url {https://arxiv.org/abs/1711.08939}},
  date={2017},
}

\bib{robinson1}{book}{
  author={Robinson, Abraham},
  title={Non-standard analysis},
  publisher={North-Holland},
  place={Amsterdam},
  date={1966},
  pages={xi+293},
}

\bib{SB}{article}{
  author={Sanders, Sam},
  title={To be or not to be constructive},
  journal={\emph {Indagationes Mathematicae} and arXiv \url {https://arxiv.org/abs/1704.00462}},
  date={2017},
  pages={pp.\ 68},
}

\bib{samcie18}{article}{
  author={Sanders, Sam},
  title={Some nonstandard equivalences in Reverse Mathematics},
  journal={Proceedings of CiE2018, Lecture notes in Computer Science, Springer},
  date={2018},
  pages={pp.\ 10},
}

\bib{simpson1}{collection}{
  title={Reverse mathematics 2001},
  series={Lecture Notes in Logic},
  volume={21},
  editor={Simpson, Stephen G.},
  publisher={ASL},
  place={La Jolla, CA},
  date={2005},
  pages={x+401},
}

\bib{simpson2}{book}{
  author={Simpson, Stephen G.},
  title={Subsystems of second order arithmetic},
  series={Perspectives in Logic},
  edition={2},
  publisher={CUP},
  date={2009},
  pages={xvi+444},
}

\bib{stillebron}{book}{
  author={Stillwell, John},
  title={Reverse mathematics, proofs from the inside out},
  pages={xiii + 182},
  year={2018},
  publisher={Princeton Univ.\ Press},
}

\bib{troelstra1}{book}{
  author={Troelstra, Anne Sjerp},
  title={Metamathematical investigation of intuitionistic arithmetic and analysis},
  note={Lecture Notes in Mathematics, Vol.\ 344},
  publisher={Springer Berlin},
  date={1973},
  pages={xv+485},
}

\bib{troeleke1}{book}{
  author={Troelstra, Anne Sjerp},
  author={van Dalen, Dirk},
  title={Constructivism in mathematics. Vol. I},
  series={Studies in Logic and the Foundations of Mathematics},
  volume={121},
  publisher={North-Holland},
  date={1988},
  pages={xx+342+XIV},
}


\end{biblist}
\end{bibdiv}
\bye